\documentclass[reqno]{amsproc}
\usepackage{amssymb,amsmath,amscd,euscript,mathrsfs}

\usepackage{graphicx}
\usepackage{amsfonts}

\newcommand{\f}[1]{{\mathfrak{#1}}}

\renewcommand{\det}{\mathrm{det}}

\theoremstyle{plain}
\numberwithin{equation}{section}
\newtheorem{theorem}{Theorem}[section]

\newtheorem{lemma}[theorem]{Lemma}
\newtheorem{proposition}[theorem]{Proposition}

\theoremstyle{definition}
\newtheorem{definition}[theorem]{Definition}

\newtheorem{remark}[theorem]{Remark}

\def\text#1{\;\;\;\;{\rm \hbox{#1}}\;\;\;\;}
\def\qquad{\quad\quad}

\def\lv{{\boldsymbol \lam}}

\def\msy#1{{\mathbb #1}}
\def\N{{\msy N}}

\def\Z{{\msy Z}}
\def\K{{\msy K}}

\def\R{\mathbb{R}}

\def\e{\epsilon}

\def\fa{{\mathfrak a}}
\def\fb{{\mathfrak b}}

\def\fg{{\mathfrak g}}

\def\fk{{\mathfrak k}}
\def\fl{{\mathfrak l}}

\def\fn{{\mathfrak n}}

\def\fq{{\mathfrak q}}
\def\fs{{\mathfrak s}}

\def\to{\rightarrow}
\def\Re{{\rm Re}\,}
\def\Im{{\rm Im}\,}
\def\LB{\Lambda^+(\cB )}

\def\Ad{{\rm Ad}}

\def\GL{\mathrm{GL}}
\newcommand{\SL}{\mathrm{SL}}

\def\rM{\mathrm{M}}
\def\rP{\mathrm{P}}
\def\rG{\mathrm{G}}
\def\rI{\mathrm{I}}
\def\rV{\mathrm{V}}
\newcommand{\rS}{\mathrm{S}}

\newcommand{\rO}{\mathrm{O}}

\def\ad{{\rm ad}}

\def\pr{{\rm pr}}

\def\tr{{\rm tr}\,}
\newcommand{\id}{\mathrm{id}}

\def\cB{{\mathcal B}}
\def\cC{{\mathcal C}}

\newcommand{\wKL}{\widehat{K}_L}
\newcommand{\Cos}{\mathrm{Cos}}
\newcommand{\CosL}{\mathrm{Cos}^{\lambda}}

\newcommand{\SinL}{\mathrm{Sin}^{\lambda}}

\def\rmO{{\rm O}}
\def\rmS{{\rm S}}
\def\rmU{{\rm U}}
\def\Tr{\mathrm{Tr}}

\def\SU{\rmS\rmU}
\def\SO{\rmS\rmO}

\newcommand{\so}{\mathfrak{so}}

\newcommand{\ip}[2]{\langle#1,#2 \rangle}

\def\sideremark#1{\ifvmode\leavevmode\fi\vadjust{\vbox to0pt{\vss
 \hbox to 0pt{\hskip\hsize\hskip1em%
 \vbox{\hsize2cm\tiny\raggedright\pretolerance10000
 \noindent #1\hfill}\hss}\vbox to8pt{\vfil}\vss}}}

\def\Cs{\mathscr{C c 1234}}

\def\Cal{\mathcal}

\def\R{{\mathbb R}}
\def\C{{\mathbb C}}

\def\S{{\Cal S}}
\def\F{{\Cal F}}
\def\K{{\mathbb K}}
\def\I{{\Cal I}}

\def\tr{{\hbox{\rm tr}}}

\def\Z{\mathbb Z}

\def\gnk{G_{n,k}}

\def\f0{f_0}
\def\Fc0{\varphi_0}

\def\I_k {I_{-}^{k/2}}
\def\I+k {I_{+}^{k/2}}

\def\vnk{\mathrm{V}_{n,k}}
\def\vnm{\mathrm{V}_{n,m}}

\def\cd{\stackrel{*}{\C}\!{}_{m, k}^\lam}
\def\sd{\stackrel{*}{\S}\!{}_{m, k}^\lam}
\def\cd0{\stackrel{*}{\cC}\!{}_{m, k}^\lam}
\def\sd0{\stackrel{*}{\S}\!{}_{m, k}^\lam}

\def\ncd0{\stackrel{*}{\Cs}\!{}_{m, k}^\lam}

\def\bbr{\mathbb{R}}

\def\bbh{\mathbb{H}}

\def\bbc{\mathbb{C}}

\def\tr{{\hbox{\rm tr}}}

\def\det{{\hbox{\rm det}}}
\def\vol{{\hbox{\rm vol}}}

\def\Pr{{\hbox{\rm Pr}}}

\def\gnk{\mathrm{G}_{n,k}}
\def\gnm{\mathrm{G}_{n,m}}

\def\part{\partial}
\def\intl{\int\limits}
\def\b{\beta}

\def\Gam{\Gamma}

\def\a{\alpha}
\def\om{\omega}

\def\vp{\varphi}

\def\gam{\gamma}

\def\lam{\lambda}

\def\e{\varepsilon}

\def\lv{{\mathbf \lam}}

\def\Cos{{\hbox{\rm Cos}}}

\def\det{{\hbox{\rm det}}}

\def\rad{{\hbox{\rm rad}}}
\def\ang{{\hbox{\rm ang}}}

\def\vol{{\hbox{\rm vol}}}

\def\gm{\Gamma_m}
\def\tr{{\hbox{\rm tr}}}
\def\part{\partial}
\def\intl{\int\limits}
\def\b{\beta}

\def\Gam{\Gamma}

\def\a{\alpha}

\def\Cs{\mathscr{C c 1234}}

\newcommand{\be}{\begin{equation}}
\newcommand{\ee}{\end{equation}}

\newcommand{\bea}{\begin{eqnarray}}
\newcommand{\eea}{\end{eqnarray}}
\newcommand{\Bea}{\begin{eqnarray*}}
\newcommand{\Eea}{\end{eqnarray*}}
\begin{document}

\title[The Cosine Transform ]
{Analytic and Group-Theoretic Aspects of the Cosine Transform}

\author{G. \'Olafsson}
\address{Department of Mathematics, Louisiana State University, Baton Rouge,
LA, 70803 USA} \email{olafsson@math.lsu.edu}

\author{A. Pasquale}
\address{Laboratoire de Math\'ematiques et Applications de Metz (UMR CNRS 7122), Universit\'e de Lorraine, 57045 Metz cedex 1, France} \email{angela.pasquale@univ-lorraine.fr}

\author{B. Rubin}
\address{
Department of Mathematics, Louisiana State University, Baton Rouge,
LA, 70803 USA}
\email{borisr@math.lsu.edu}

\thanks{The authors  are thankful  to Tufts University for the hospitality and  support during the Joint AMS meeting and the Workshop on Geometric Analysis on Euclidean and Homogeneous Spaces in January,  2012.
 The  research of G. \'Olafsson  was supported  by DMS-0801010 and DMS-1101337.  A. Pasquale gratefully acknowledges travel support from the Commission de Colloques et Congr\`es Internationaux (CCCI)}

\begin{abstract}
This is a brief survey of recent results by the authors devoted to one of the most important operators of integral geometry. Basic facts about the analytic family of cosine transforms on the unit sphere in $\bbr^n$ and  the corresponding Funk transform are extended to the ``higher-rank" case for  functions on Stiefel and Grassmann manifolds. The main topics are the
analytic continuation and the structure of polar sets, the connection with the Fourier transform on the space of rectangular matrices, inversion formulas and spectral analysis, and the group-theoretic realization as an intertwining operator between generalized principal series representations of $SL(n, \bbr)$.
\end{abstract}

\maketitle

\section{Introduction}

\setcounter{equation}{0}
\noindent
The \textit{cosine transform}  has a long and rich
history, with connections to several branches of mathematics.
  The name {\it cosine transform}  was  adopted by   Lutwak
   \cite[p. 385]{Lu}  for the spherical convolution which is defined on the unit sphere $\rS^{n-1}$ in $\bbr^n$ by 
\be\label{t11}(\cC f)(u)=\int_{\rS^{n-1}} f(v) |u \cdot v| \, dv, \qquad u \in
\rS^{n-1}\, .\ee
 The motivation for this name is that the  inner product $u \cdot v$ is nothing but the cosine of the angle between the  unit vectors $u$ and $v$.

The following  list of references shows some branches of mathematics, where  the operator (\ref{t11})  and its generalizations arise in a natural way (sometimes implicitly,  without naming)  and play an important role.
\medskip

\noindent
$\bullet$ {\it Convex geometry:} \cite{Ale, Bla, Ga, GG, Gro, Ko2, Lu, Ru08, RZ, Schn}.
\smallskip

\noindent
$\bullet$ {\it Pseudo-Differential Operators}: \cite {Es, Pla}.
\smallskip

\noindent
$\bullet$ {\it Group representations}: \cite {A, AB, DH, DM, OP, Pa}.
\smallskip

\noindent
$\bullet$ {\it Harmonic Analysis and Singular Integrals}: \cite{AW, Gad82, Gad89, GSha, Kry, MP, OR05, OR06, Ru98a, Ru02, Sa80, Sa83, Str70}.
\smallskip

\noindent
$\bullet$ {\it Integral geometry}: \cite{Be, Fu, GGR, GYY, Ru98b, Ru99b, Ru99a, Ru03,  Ru12, Se,   Zh09}.
\smallskip

\noindent
$\bullet$ {\it Stochastic Geometry and Probability}: \cite {GH, Le, Mat, Sp01, Sp02}.
\smallskip

\noindent
$\bullet$ {\it Banach Space Theory}: \cite {Kan, Ko1, KK, Ney, Rud}.
\medskip

This list is far from being complete. In most of the publications cosine-like transforms serve as a tool for certain specific problems. At the same time,
 there are many papers devoted to the cosine transforms themselves. The present article is just of this kind. Our aim is to give a short overview of our recent  work \cite{OP, Ru12} on the cosine transform and explain some of the ideas and tools
behind those results.

  For a complex number  $\lam$, the  $\lam$-analogue of the operator (\ref{t11}) is the convolution
operator
 \be\label{t11la}(\cC^\lam f)(u)=\int_{\rS^{n-1}} f(v) |u \cdot v|^\lam \, dv, \qquad u \in \
\rS^{n-1},\ee
where  the integral is understood in the sense of analytic continuation, if necessary.
 We adopt the name ``the cosine transform" for  (\ref{t11la}) too. The same name will be used  for  generalizations of these operators to be defined below.

  In recent years more general, higher-rank cosine transforms  attracted considerable attention. This class of operators was
   inspired by    Matheron's injectivity conjecture \cite{Mat}, its disproval by Goodey and Howard \cite{GH},
  applications in group representations \cite{BOO, DM, OP, Pa, Zh09} and in algebraic integral geometry \cite{AB, Be, Fu}.
   To the best of our knowledge,  the higher-rank cosine transform was explicitly presented (without naming) for the first time in \cite [formula (3.5)]{GGR}.

As mentioned above, the present article gives a brief survey of recent results by the authors \cite{OP, Ru12} in this area.
The consideration grew up from specific problems of harmonic analysis and group representations. However, we do not focus on those problems, and mention them only for  better explanation of the corresponding properties of the cosine transforms and related operators of integral geometry.
Here we shall restrict ourselves to the case of real numbers, referring to the above articles for the case of complex and
quaternionic fields.

The paper is organized as follows. Section 2 contains  basic facts about the  cosine transforms on the unit sphere.  More general higher-rank  transforms on Stiefel or Grassmann manifolds  are considered in Section 3, where the main tool is the classical Fourier analysis. In Sections 4 and 5 we discuss the connections to representation theory, and more precisely to the spherical representations and the intertwining properties. Section 6 is devoted to explicit spectral formulas for the cosine transforms.

\section{Cosine transforms on the unit sphere}

\noindent
In this section we discuss briefly the cosine transform on the sphere $\rS^{n-1}$. We keep the
notation from the Introduction. For the analytic continuation of the cosine transform it is convenient to
normalize it by setting
 \be\label{af}    (\Cs^\lam
f)(u)= \gam_n(\lam)   \int_{\rS^{n-1}}  \!  \!  \! f(v)
|u \cdot v|^{\lam} \,dv,\qquad u \in
\rS^{n-1}. \nonumber\,  \ee
Here $dv$ stands for the $\SO (n)$-invariant probability measure on $\rS^{n-1}$ and
the normalizing coefficient $\gamma_n (\lambda)$ is given by
\be\label{beren}
\gam_n(\lam)\!=\!\frac{\pi^{1/2}\,\Gamma( -\lam/2)}{\Gamma (n/2)\, \Gamma ((1+\lam)/2)}, \qquad \Re \, \lam \!>\!-1, \quad \lam
\!\neq \!0,2,4
, \ldots .\ee
This normalization is chosen so that
\[\Cs^\lambda (1)=\frac{\Gamma\left(-\lambda /2\right)}{ \Gamma ((n+\lam)/2)}\,.\]
Such a normalization is convenient in many occurrences, when  harmonic  analysis on the sphere is performed in the multiplier language (in the same manner as analysis of pseudo-differential operators is  performed in the language of their symbols).
We shall see below that it also simplifies the spectrum of the cosine transform.

The limit case $\lam=-1$ gives, up to a constant, the well-known Funk transform.
Specifically, if $f\in C (\rS^{n-1})$, then for  every $ u \in
\rS^{n-1}$,
\be\label{lim1}
\lim\limits_{\lam \to -1} (\Cs^\lam f)(u)= \frac{\pi^{1/2}}{\Gamma ((n-1)/2)}\, (Ff)(u),
\ee
where
\be\label{Funk}
 (Ff)(u)=\!\!\!
 \intl_{\{v\in S^{n-1} \mid
 u \cdot v =0\}} \!\!\!\! f(v) \,d_u v\,.
\ee
In (\ref{Funk}),   $d_u v$ stands for the rotational invariant probability measure on the $(n-2)$-dimensional sphere  $u \cdot v =0$; see, e.g., \cite[Lemma 3.1]{Ru08}.

 The operators $\cC^\lam$ and  $\Cs^\lam$ were
  investigated  by different approaches.
A first one employs the Fourier transform technique \cite {Ko1, Ru98a, Se}
 and relies on the equality in the sense of distributions
  \be\label{eq5}
 \left(\frac{E_{\lam}\, \cC^{\lam} f}{\Gamma ((1+\lam)/2)}, \F \om\right)=c_1 \, \left (\frac{E_{-\lam-n}  f}
 {\Gam (-\lam/2)},  \om\right),\ee
\[ c_1 =2^{n+\lam}\, \pi^{(n-1)/2}\,\Gam (n/2).\]
Here  $\om$ is a test function belonging to the Schwartz space
 $S(\bbr^n)$, $$(\F \om)(y)=\int_{\bbr^n} \om (x) e^{i x\cdot y}dx,$$ and $(E_\lam f)(x)=|x|^{\lam} f(x/|x|)$ denotes the extension by homogeneity.

A second approach is based on  the Funk-Hecke
formula, so that for each  spherical
harmonic  $Y_j$ of  degree $j$,
\be\label{sh} \Cs^\lam Y_j=m_{j, \lam} \, Y_j,\ee
where \be\label{55} m_{j, \lam} \!=\!\left\{
\begin{array}{cl} \!(-1)^{j/2}\, \displaystyle{\frac{\Gamma
(j/2-\lam/2)}{ \Gamma (j/2+(n+\lam)/2)}}
  &  \mbox{\rm if $j$ is even}, \\
0 &  \mbox{\rm if $j$ is odd};
\end{array}
\right.  \ee
see, e.g., \cite{Ru98a}. The Fourier-Laplace multiplier $\{m_{j, \lam}\}$ forms the spectrum of $\Cs^\lam $. Note that the normalizing coefficient in   $\Cs^\lam$ was chosen so that  only factors depending on $j$
are involved in the spectral functions $\{m_{j, \lam}\}$. The spectrum of $\Cs^\lam $  encodes important information about this operator.
 For instance, since $m_{j, \lam} m_{j, -\lam-n}=1$,  then for any $f\in C^\infty (\rS^{n-1})$ the following inversion formula holds:
\be\label{90r}
\Cs^{-\lam -n} \Cs^\lam f=f,
\ee
provided
\[\lam \in \bbc, \qquad\lam \notin \{-n, -n-2, -n-4, \ldots\}\cup\{0,2,4,\ldots\}.\]
For the non-normalized transforms,  (\ref{90r}) yields
\be\label{90ra}
\cC^{-\lam -n} \cC^\lam f=\zeta (\lam)\, f, \qquad \zeta (\lam)=\frac{\Gamma^2 (n/2)\, \Gamma ((1+\lam)/2)\, \Gamma ((1-\lam-n)/2)}{\pi\, \Gamma( -\lam/2)\,\Gamma ((n+\lam)/2)},
\ee
\[\lam \in \bbc, \qquad\lam \notin \{-1, -3, -5, \ldots\}\cup\{1-n,  3-n,  5-n, \ldots\}.\]

Formula   (\ref{55}) reveals  singularities, provides information about the  kernel and the image. Moreover, it plays a crucial
role in the study of cosine transforms on $L^p$ functions. For instance, the following statement was proved in \cite [p.11] {Ru99b}, using the relevant results
of  Gadzhiev \cite  {Gad82, Gad89}  and Kryuchkov \cite{Kry} for symbols of the Calderon-Zygmund singular integrals operators.

\begin{theorem}  Let $L^p_e(\rS^{n-1})$ and $L^\gamma_{p, e}(\rS^{n-1})$ be the
spaces of even functions (or distributions), belonging to $L^p(\rS^{n-1})$ and the Sobolev space $L^\gamma_p (\rS^{n-1})$,
respectively. Then
\be\label{907} L^\delta_{p,e}(\rS^{n-1}) \subset \Cs^\lam (L^p_e (\rS^{n-1}))
\subset L^\gamma_{p,e}(\rS^{n-1})\ee
 provided
$$ \gamma = \Re  \lam + \frac{n +1}{2} - \Big |
\frac1{p} - \frac12 \Big | (n -1),
\qquad \delta = \Re \lam +
\frac{n +1}{2} + \Big | \frac1{p} - \frac12
\Big | (n - 1),
$$
\centerline{$\lam \notin \{ 0, 2, 4, \dots \} \cup \{ - n-1, - n - 3, -n - 5,
\dots \}.$}
The embeddings  (\ref{907}) are sharp.
\end{theorem}

Finally, one can use tools from the representation theory, as we will discuss in more details in the second half of this
article.

One can easily explain   (\ref{sh}) -- but not (\ref{55}) -- by the fact that the space of harmonic polynomials
of degree $j$ is the underlying space of an irreducible representation of $K=\SO (n)$. Then (\ref{sh}) follows
from Schur's
lemma and the fact that $\Cs^\lam$ commutes with rotations. Note that the  group
$K$ acts by the left
regular representation on $L^2(S^{n-1})$ and, as a representation of $K$, we have
the orthogonal decomposition
\begin{equation}\label{eq:Sphere}
L^2(S^{n-1}) \simeq_K \bigoplus_{j\in \N_0} \mathcal Y^j,
\end{equation}
where the set $\mathcal Y^j$ of all spherical harmonics of  degree $j$
is an irreducible $K$-space. As we shall see in Section \ref{section:evaluation}, the spectral multiplier
(\ref{55}) can also be computed by identifying $\Cs^\lam$ as a standard intertwining
operator between certain
principal series representations of the larger group $\SL(n,\bbr)$, see \cite{OP}.

The fact that $\Cs^\lam$ is zero on the odd power harmonics follows from the observation that the  kernel
$|u\cdot v|^\lambda$ is an even function of $v$. Hence $\Cs^\lam$ is actually  an integral transform on the
projective space $\rP (\R^n)$, and here the analogue of (\ref{eq:Sphere}) is
\[
L^2(\rP (\R^n)) \simeq_K \bigoplus_{j\in 2\N_0} \mathcal Y^j\, .\]

\section{Cosine transforms on Stiefel and Grassmann manifolds}
\noindent
In this section we introduce the higher-rank cosine transforms and collect some basic facts about these
transforms. The main results are presented in Theorems \ref{lhgn1}, \ref{th:33},  \ref{lhgn}, \ref{cr72}, and \ref {cr24n}.

\subsection{Notation}
\label{section:notation}
We denote by $\vnm \sim \rO(n)/\rO(n-m)$   the Stiefel manifold  of $n\times m$ real matrices, the columns of which
 are mutually orthogonal unit $n$-vectors. For $v\in \vnm$,  $dv$ stands for the  invariant probability measure on $\vnm$; $\xi=\{v\}$ denotes the linear subspace of $\bbr^n$ spanned by $v$.  These subspaces form the  Grassmann manifold $\rG_{n,m}\sim \rO(n)/(\rO(n-m) \times \rO(m))$ endowed with the  invariant probability measure $d\xi$.
  We write  $\rM_{n,m}\sim\bbr^{nm}$ for  the
space of real matrices $x=(x_{i,j})$ having $n$ rows and $m$
 columns and set   $$dx=\prod^{n}_{i=1}\prod^{m}_{j=1} dx_{i,j}, \qquad |x|_m=\det
(x^tx)^{1/2}, $$
 $x^t$ being the transpose of $x$.
 If $n=m$,  then $|x|_m$ is just the absolute value of the determinant of $x$;   if
$m=1$, then $|x|_1$ is the usual Euclidean norm of $x\in \bbr^n$.

\subsection{The $\Cos$-function} We  give two equivalent  ``higher-rank" substitutes for  $|u\cdot v|$ in (\ref{t11}). The first one is ``more geometric", while the second is ``more analytic".
For $1\le m\le k\le n-1$, let $\eta \in \rG_{n,m}$ and $\xi \in \rG_{n,k}$   be linear subspaces of $\bbr^n$ of dimension $m$ and $k$, respectively.
Following \cite{A, AB, OP}, we set
\be\label{pr} \Cos (\xi, \eta)=\vol_m (\Pr_{\xi} E),\ee
where $\vol_m (\cdot )$ denotes the $m$-dimensional volume function, $E$  is a  convex subset  of $\eta$ of volume one   containing the origin, $\Pr_{\xi}$  denotes the orthogonal projection
 onto $\xi$.
By affine invariance, this definition is independent of the choice of $E$.

The second definition \cite{GR}  gives precise meaning to the projection operator $\Pr_{\xi}$.
Let $u$ and $v$ be arbitrary orthonormal bases  of
 $\xi$ and  $\eta$, respectively. We regard $u$ and $v$ as elements of the corresponding
 Stiefel manifolds  $\vnk$ and   $\vnm$. If $k=m=1$, then $u$ and $v$ are unit vectors, as in (\ref{t11}).
 The orthogonal projection   $\Pr_{\xi}$ is given by the $k\times k$ matrix $uu^t$, and we can define
 \be\label{pr1} \Cos (\xi, \eta )\equiv \Cos (\{u\}, \{v\}) =(\det (v^tuu^tv))^{1/2}\equiv |u^tv|_m.\ee
 This definition is independent of the choice of bases  in  $\xi$ and $\eta$ and yields $|u\cdot v|$ if $k=m=1$.

\begin{remark} Note that $v^tuu^tv$ is a positive semi-definite  matrix, and therefore, 
$\det (v^tuu^tv)\equiv \det (u^tvv^tu)\ge 0$. It means that $\Cos (\xi, \eta)=\Cos (\eta, \xi)\ge 0$.
\end{remark}

\subsection{Non-normalized cosine transforms}

According to  (\ref{pr}) and (\ref{pr1}),  one can use both Stiefel and Grassmannian language in the definition of the higher-rank cosine transform, namely,
\be\label{cos1}(\cC^{\lam}_{m, k} f)(u)=\int_{\vnm}  f(v)\,
|u^tv|_m^{\lam} \, dv,    \qquad    u \!\in\! \rV_{n,k},\ee
\be\label{cos2}(\cC^{\lam}_{m, k} f)(\xi)=\int_{\gnm} f(\eta )\,
\Cos^{\lam} (\xi , \eta ) \, d\eta,  \qquad    \xi \!\in\! G_{n,k},\ee
where $dv$ and $d\eta$ stand for the relevant invariant probability measures. The fact that we have two ways of writing of the same operator, extends the arsenal of  techniques (some of them will be exhibited below).
 Both operators agree with $\cC^{\lam}$ in  (\ref{t11la}), when  $k=m=1$.
For brevity, we shall write $\cC^\lam_m=\cC^\lam_{m,m}$.

We remark that there are different shifts in the power $\lambda$ in the literature, all for  different reasons. In particular, to make our statements in Sections 2-4 consistent with those in \cite{Ru12},  one should set $\lam=\a -k$.
To adapt to the notation in \cite{OP} one has to change $\lambda$ to $\lambda -n/2$.
For unifying the presentation of the results in \cite{Ru12} and \cite{OP} we have preferred to adopt the
unshifted notation as in (\ref{cos1}) and (\ref{cos2}).

Following \cite{FK, Gi}, the  Siegel gamma  function of the cone $\Omega$ of positive definite $m \times m$ real
symmetric matrices is defined by
\be\label{2.4}
 \gm (\a)\!=\!\int_{\Omega} \exp(-\tr (r)) |r|_m^{\a-(m+1)/2 } dr
 =\pi^{m(m-1)/4}\prod\limits_{j=0}^{m-1} \Gam (\a\!-\! j/2)  \ee
 and represents a meromorphic function with
 polar set  \be\label {09k}\{(m-1-j)/2\, \mid \,  j=0,1,2,\ldots\}.\ee

\begin{theorem}\label{lhgn1}
Let  $1\le m\le k\le n-1$.

\begin{enumerate}
\renewcommand{\theenumi}{\roman{enumi}}
\renewcommand{\labelenumi}{(\theenumi)}

\item If  $f \!\in \!L^1(\vnm)$ and   $\Re \lam>m-k-1$, then the integral  (\ref{cos1}) converges for almost all $u \!\in\! V_{n,k}$.

\item If  $f\in C^\infty (\vnm)$, then for every $u \!\in\! V_{n,k}$, the function $ \lam \mapsto (\cC^{\lam}_{m, k} f)(u)$ extends to the domain
 $\Re \lam \leq m-k-1$
as a meromorphic function  with the only poles $ \;
m-k-1, m-k-2,\dots\;$. These poles and their orders are  the
same as of the gamma function $\gm((\lam +k)/2)$.

\item The normalized
integral $ (\cC^{\lam}_{m, k} f)(u)/\gm((\lam +k)/2)$
 is an entire
function of $\lam$ and belongs to $ C^\infty (\vnk)$ in the $u$-variable.

\end{enumerate}

\end{theorem}

A similar statement holds for (\ref{cos2}).
The proof of Theorem \ref {lhgn1} can be found in \cite[Theorems 4.3, 7.1]{Ru12}. It relies on the fact that
$|u^tv|_m^{\lam}$ is a special case of the composite power function $(u^tv)^{\lv}$ with the vector-valued exponent $\lv\in \bbc^m$ \cite{FK, Gi}. The corresponding {\it composite cosine transforms} were studied in \cite{OR05, OR06, Ru12}.

An important ingredient of the proof of Theorem \ref {lhgn1} is the
 connection between  the cosine transform $\cC^{\lam}_{m, k} f$ on $\vnm$ and the Fourier transform
  \be\label{ft}
\hat\vp (y)=(\F\vp)(y)=\int_{\rM_{n,m}} e^{{\rm tr}(iy^t x)} \vp
(x)\, dx,\qquad y\in \rM_{n,m} \; .\ee
The corresponding Parseval
equality
 has the form \be\label{pars} (\hat \vp, \hat \om)=(2\pi)^{nm} \, (\vp, \om),
\qquad (\vp, \om)=\int_{\rM_{n,m}} \vp(x) \overline{\om (x)} \,
dx.\ee This equality with $\om$ in the Schwartz class $S(\rM_{n,m})$ of smooth rapidly decreasing  functions is used to define
the Fourier transform  of the corresponding  distributions.

We will need polar coordinates on $\rM_{n,m}$, so that  for  $n  \ge  m$,  every matrix  $x \in  \rM_{n,m}$ of rank $m$  can be uniquely represented as $ x=vr^{1/2}$ with $ v \in \vnm$ and   $r=x^tx \in\Omega$. Given a function $f$ on $\vnm$,
we denote $(E_\lam f)(x)=|r|_m^{\lam/2} f(v)$.
The following statement holds in the case $k=m$.
\begin{theorem}\label{th:33} Let $f$ be an integrable right $\rO(m)$-invariant function on $\vnm$, $\om\in
 \S(M_{n,m})$, $1\!\le\! m\! \le\! n\!-\!1$,  $\cC^{\lam}_{m} f=\cC^{\lam}_{m, m} f$. Then for every $ \lam\in\bbc$,
 \be\label{eq5ya}
 \left(\frac{E_{\lam} \cC^{\lam}_{m} f}{\Gam_m ((\lam +m)/2)}, \F \om\right)=c \, \left (\frac{E_{-\lam-n}  f}{\Gam_m (-\lam/2)},  \om\right),\ee
\[ c =\frac{2^{m(n+\lam)}\, \pi^{nm/2}\,\Gam_m (n/2)}{\Gam_m (m/2)},\]
where both sides   are understood in the sense of analytic continuation.
\end{theorem}
The formula (\ref{eq5ya}) agrees with (\ref{eq5}).
The more general statement for arbitrary $k\ge m$ can be found in \cite{Ru12}.

\begin{remark} {\rm It is important to note  that the
 domains, where the left-hand side and
the right-hand side  of  of (\ref{eq5ya}) exist as absolutely
convergent integrals, have no points in common, when  $m>1$. This is the principal distinction from the case $m=1$, when there is a common strip of convergence $-1<\Re \lam<0$.
To  perform analytic continuation, we have to  switch from  $\cC^{\lam}_{m}$ to the more general composite cosine transform $\cC^{\lv}_{m}$ with  $\lv\in \bbc^m$ and then take the restriction to the diagonal $\lam_1 =\cdots =\lam_m=\lam +m$.  This method of analytic continuation was first used by  Kh\`ekalo (for another class of operators) in his papers \cite{Kh1, Kh1a, Kh2} on Riesz potentials on the
space of rectangular matrices.}
\end{remark}

\subsection{The Funk transform}

 The  higher-rank version of the classical Funk transform (\ref{Funk})  sends a function $f$ on $\rV_{n, m}$ to a function $F_{m,k} f$ on $\rV_{n, k}$ by the formula
 \be \label {la3v}(F_{m,k} f) (u)=\int_{\{v\in \rV_{n, m}\mid \, u^tv=0\}} f(v)\,d_u v,
\qquad u\!\in\! V_{n, k}.\ee
The condition $u^tv=0$ means that subspaces $\{u\} \in \gnk$ and $\{v\} \in \gnm$ are mutually orthogonal. Hence,    necessarily,  $k+m\le n$. The case $k=m$, when both $f$ and its Funk transform live on the same manifold, is of particular importance and coincides with (\ref{Funk}) when $k=m=1$. We denote $F_m= F_{m,m}$.

If $f$ is right $\rO(m)$-invariant, $(F_{m,k} f) (u)$ can  be identified with a function on the Grassmannians  $\rG_{n,m}$
 or $\rG_{n,n-m}$, and  can be  written ``in the Grassmannian language". For instance, setting $\xi\!=\!\{v\}\!\in \!\rG_{n,m}$, $\eta\!=\!\{u\}^\perp\!\in\! \rG_{n,n-k}$, and
$ \tilde f(\xi)= f(v)$,
we obtain
\be (R_{m,n-k} \tilde f)(\eta)\!\equiv\!\intl_{\xi
\subset \eta} \!\!\tilde f(\xi)\, d_\eta \xi=(F_{m, k}f)(u).\ee

\subsection{Normalized cosine transforms}

Our next aim is to introduce a natural generalization $\Cs_{m,k}^\lam f$ of the normalized  transform (\ref{af}). ``Natural" means that we expect $\Cs_{m,k}^\lam f$ to obey   the relevant higher-rank modifications  of the properties (\ref{lim1})-(\ref{sh}).

\begin{definition}  Let $1\le m\le k \le n-1$.
For $u \!\in\! V_{n,k}$ and $ v \!\in\! V_{n,m}$, we define
\be \label{n0mby}(\Cs_{m,k}^\lam f)(u)=\gam_{n,m,k} (\lam)\int_{\vnm} \!\!\!f(v)\,
|u^tv|_m^{\lam} \, dv,\ee
where
\be\gam_{n,m,k}
(\lam)=\frac{\Gam_m(m/2)}{\Gam_m(n/2)}\,\frac{\Gam_m(-\lam/2)}{\Gam_m((\lam +k)/2)}, \qquad \lam+m\neq 1,2, \ldots\, .
\nonumber \ee
\end{definition}
 We denote $\Cs_{m}^{\lam}=\Cs_{m,m}^{\lam}$. The integral (\ref{n0mby}) is absolutely convergent if $\Re \, \lam
 > m-k-1$. The excluded values of $\lam$ belong to the polar set  of $\Gam_m(-\lam/2)$.
If $k=m=1$ this definition coincides with (\ref{af}). Operators of this kind implicitly arose in
\cite [pp. 367, 368]{GGR}.

\begin{theorem}\label{lhgn}
Let  $1\le m\le k\le n-1$, $ k+m\le n$. If  $f$ is a $C^\infty$ right
$\rO(m)$-invariant function on $\vnm$, then for every $u \!\in\! \rV_{n,k}$,

  \be\label{zn0x} \underset
{\lam=-k}{a.c.} \,(\Cs^{\lam}_{m, k} f)(u)=\frac{\gm(m/2)}{\gm((n-k)/2)}\,
(F_{m, k}f) (u), \ee
where ``$a.c.$" denotes analytic continuation and $(F_{m, k}f) (u)$ is the Funk transform (\ref{la3v}).
\end{theorem}
This statement follows from \cite [Theorems 7.1 (iv) and 6.1]{Ru12}. Note that if $m=k=1$, then (\ref{zn0x}) yields (\ref{lim1}). However, unlike (\ref{lim1}), the proof of which is straightforward, (\ref{zn0x}) requires a certain indirect procedure, which invokes the Fourier transform on the space of matrices and the relevant analogue of (\ref{eq5ya}).

We point out that a pointwise inversion of the Funk transform can be obtained by means of the dual cosine transform, which is defined by
\be\label{0mbyd}
(\cd0\vp)(v)=\int_{\vnk} \!\!\vp(u)\, |u^tv|_m^{\lam} \, du, \qquad v\in \vnm.\ee
Indeed, the following result holds.
\begin{theorem} \label{cr72} {\rm (cf.  \cite [Theorems 7.4]{Ru12})} \ Let $\vp=F_{m,k} f$, where $f$ is a $C^\infty$ right
$\rO(m)$-invariant function on $\vnm$,  $1\le m\le k\le n-m$. Then, for every $v\in \vnm$,
\be \label {forq1y}  \underset
{\lam=m-n}{a.c.} \frac{(\cd0 \vp)(v)}{\Gam_m((\lam +k)/2)}=c\, f(v), \qquad c \!=\!\frac{\gm(n/2)}{\Gam_m (k/2)\, \gm(m/2)}.\ee
\end{theorem}

Regarding other inversion methods of the higher-rank Funk transform (which is also known as the Radon transform for a pair of Grassmannnians), see \cite{GR, Zh1} and references therein.

In the case $k=m$ the normalized cosine transform
$\Cs_{m}^{\lam}=\Cs_{m,m}^{\lam}$ has a number of important features.  If $f\in C^\infty (\vnm)$, then analytic continuation of  $(\Cs_{m}^{\lam}f)(u)$ is well-defined for all complex $\lam\notin \{1-m, 2-m, \ldots \}\,$ and belongs to $C^\infty (\vnm)$. The following inversion formulas hold.

\begin{theorem} \label{cr24n}  {\rm (cf.  \cite [Theorems 7.7]{Ru12})} \  Let $f\in C^\infty (\vnm)$ be a  right
$\rO(m)$-invariant function on $\vnm$, $2m \le n$. Then, for every $u\in \vnm$,
\be \label {for1n}
(\Cs_m^{ -\lam -n} \Cs_m^{\lam} f)(u)=f(u), \qquad \lam,  -\lam -n \notin \{1-m, 2-m, \ldots \}.\ee
In particular, for the non-normalized transforms,
\be \label {for1na}
( \cC_m^{ -\lam -n} \cC_m^{\lam} f)(u)=\zeta (\lam)\,f(u), \qquad \lam +n, -\lam \notin \{1, 2, 3, \ldots\},\ee
where
\be \label{oouy} \zeta (\lam)=\frac{\Gamma_m^2 (n/2)\, \Gamma_m ((m+\lam)/2)\, \Gamma_m ((m-\lam-n)/2)}{\Gamma_m^2 (m/2)\, \Gamma_m( -\lam/2)\,\Gamma_m ((n+\lam)/2)}.
\ee
Both equalities (\ref{for1n}) and (\ref{oouy}) are understood in the sense of analytic continuation.
\end{theorem}

In the case $m=1$, the formulas (\ref{for1n}) and (\ref{for1na}) coincide with (\ref{90r}) and (\ref{90ra}), respectively, but the method of the proof is different.

\section{Connection to Representation Theory}
\label{se:ConnRep}

\noindent
The cosine transform is closely related to the representation theory of semisimple Lie groups. In particular, as we shall now discuss, it has an important group-theoretic interpretation as a standard intertwining operator between generalized principal series representations of $\SL(n,\R)$.

In the following we shall use the notation $G=\SL (n,\R)$,  $K=\SO (n)$, and
\[L=\rS (\rO (m )\!\times\! \rO (n-m))\! =\! \left\{\left. \begin{pmatrix} A & 0 \\
0 & B\end{pmatrix}\, \right| \; \begin{matrix} A\in \rO (m) \hfill\\  B\in\rO (n-m)\end{matrix} \,, \;\,
\det(A)\det (B)=1\right\}\,\]
with $m \leq n-m$.
Then
$\cB\equiv K/L=G_{n,m}$ is the Grassmanian of $m$-dimensional linear subspaces of $\R^n$.   We fix the base point
\[b_o=\{(x_1,\ldots ,x_m,0,\ldots ,0 )\mid x_1,\ldots ,x_m\in \R\}\in\cB,\, \]
so that $\cB=K \cdot b_0$ and every function on $\cB$ can be regarded as a right $L$-invariant function on $K$.

 From now on, our main concern is the nonnormalized cosine transform  (\ref{cos1})  with equal lower indices, that is, $\cC^\lambda_m\equiv\cC^{\lambda}_{m,m}$.  We refer
to \cite[Chapter V]{GGA} for harmonic analysis on compact symmetric spaces and \cite{Knapp86}
for the representation theory of semisimple Lie groups.

\subsection{Analysis on $\cB$ with respect to $K$}

The first connection to representation theory is related to the left regular action of the group $K$ on $L^2(\cB)$ by
$$\big(\ell(k)f\big)(b)=f(k^{-1}b),\qquad k\in K\,, \; b\in\cB.$$
For an  irreducible unitary representation $(\pi,V_\pi)$ of $K$, we consider the subspace
\[V_\pi^L:=\{v\in V_\pi \mid  \, \pi (k)v=v\; \forall k\in L\},\qquad L=\rS (\rO (m )\!\times\! \rO (n-m)).\]
The representation $(\pi,V_\pi)$ is said to be
\textit{$L$-spherical} if  $V_\pi^L\not= \{0\}$. As $\cB=K/L$ is a symmetric space, the following result is a consequence of  \cite[Chapter IV, Lemma 3.6]{GGA}.

 \begin {proposition}
If $(\pi ,V_\pi)$ is $L$-spherical, then  $\dim V_\pi^L=1$.
\end{proposition}

Since $V_{\pi}^{L}\not= \{0\}$, we can  choose a unit vector $e_\pi \in V_\pi^L$. Then we define a map
 $\Phi_\pi :V_\pi \to C^\infty (\cB)\subset L^2(\cB)$ by the formula
\begin{equation}\label{eq:EmbL2}
 (\Phi_\pi v)(b):= d (\pi )^{-1/2}\ip{v}{\pi (k) e_\pi}\,, \qquad v\in V_\pi, \quad b=k \cdot b_o\in \cB=K\cdot b_o,
\end{equation}
where $d(\pi )=\dim V_\pi$. This definition is meaningful because $k \cdot b_o=kk' \cdot b_o$ for every $k'\in L$ and $e_\pi$ remains fixed under the action of $\pi (k')$. We also set
\[\Phi_\pi (v;b):=(\Phi_\pi v)(b).\]
Recall, if $(\pi,V_\pi)$ and $(\sigma ,V_\sigma)$ are two representations of a Hausdorff topological group $H$,  then an intertwining operator between $\pi$ and $\sigma$ is a bounded linear operator
$T: V_\pi \to V_\sigma$ such that $T\pi (h)=\sigma (h)T$  for all $h\in H$. If $\pi$ is irreducible
and $T$ intertwines $\pi$ with itself, then Schur's Lemma states that $T=c \,\id$ for some
complex number $c$, \cite{Fo}, p. 71.
The map $\Phi_\pi$ is a $K$-intertwining operator in the sense that it intertwines the representation $\pi$ on $V_\pi$  and the left regular representation  $\ell$ on $L^2(\cB)$, so that for $b=h\cdot b_o$ and
$k\in K$ we
have
\[  \Phi_\pi( \pi (k)v;b)=
\ip{\pi (k)v}{\pi (h)e_\pi}=
\ip{v}{\pi (k^{-1}h)e_\pi}=\ell (k)\Phi_\pi (v;b)\, . \]
Furthermore, the left regular representation $\ell$ on $L^2(\cB)$ is
multiplicity free, see e.g. \cite[Corollary 9.8.2]{Wolf}. Therefore, since
$(\pi,V_\pi)$ is irreducible, any intertwining operator $V_\pi \to L^2(B)$
is by Schur's Lemma of the form $c\, \Phi_\pi$ for some $c \in \C$.

We let $L^2_\pi (\cB )=\Im \Phi_\pi$.
Denote by $\widehat{K}_L$ the set of all equivalence classes of irreducible
$L$-spherical representations $(\pi,V_\pi)$ of $K$. Then, see \cite[Chapter V, Thm. 4.3]{GGA}, the decomposition of $L^2(\cB)$ as a $K$-representation is as follows.
\begin{theorem}
$\displaystyle L^2(\cB )\simeq_K\bigoplus_{\pi\in \widehat{K}_L} L^2_\pi (\cB )\, .$
\end{theorem}

The cosine transform is, as mentioned before, a $K$-intertwining operator, i.e., $\cC^\lam_m (\ell(k)f)= \ell(k)\cC^\lam_m(f)$ for all
$k \in K$ and $f \in L^2(\cB)$. It follows by Schur's Lemma that for each $\pi \in \widehat{K}_L$ there exists a
function $\eta_\pi$ on $\C$ such that
\begin{equation}\label{eq:defEta}
\cC^\lam_m|_{L^2_\pi}= \eta_\pi (\lambda) \, \id\, .
\end{equation}
Let $f\in L^2_\pi (\cB)$ of norm one. Then $\eta_\pi (\lambda )=\langle\cC^\lambda (f),f\rangle$ and
it follows that $\eta_\pi (\lambda)$ is meromorphic; cf. Theorem \ref {lhgn1}.

\subsection{Generalized spherical principal series representations of $G$}
\label{subsection:principal-series}
The fact that $\cC^\lam_m$ is a $K$-intertwining operator does not indicate \textit{how to determine} the
functions $\eta_\pi$. In the case $m=1$ and in some particular cases for the higher-rank cosine transforms \cite{OR05, OR06} explicit expression for  $\eta_\pi$ can be obtained using the Funk-Hecke Theorem or the Fourier transform technique. It is a challenging open problem to proceed the same way in the most general case, using, e.g., the relevant results of Gelbart, Strichartz, and Ton-That, see, e.g., \cite{Ge, Str75, TT}. Below we suggest an alternative way  and proceed  as follows.

To find $\eta_\pi$ explicitly,  we observe  that the cosine transform is an intertwining operator between certain
generalized principal series representations $(\pi_\lambda, L^2(\cB))$ of $G=\SL (n,\R)$ induced from a maximal parabolic subgroup of $G$. We can then use the bigger group $G$, or better its Lie algebra, to move between $K$-types.
We invoke the \textit{spectrum generating technique} introduced in
\cite{BOO} to build up a recursion relation between the spectral functions $\eta_\pi$. This finally allows us to
determine all of them by knowing $\eta_{\rm{trivial}}$.

The group $G=\SL(n,\R)$ acts on $\cB$ by
\[g\cdot \eta := \{gv\mid v\in \eta\}\,,\]
where $gv$ denotes the usual matrix multiplication.
This action is  transitive, as the $K$-action is already transitive. The stabilizer of $b_o$
is the group
\begin{eqnarray*}
P&=&\left\{\left. \begin{pmatrix} A & X\\ 0 & B\end{pmatrix}\, \right|\, X\in \rM_{m,n-m}\,, \;
\begin{matrix} A\in \GL (m,\R) \hfill\\
B\in \GL(n-m,\R)\end{matrix}\text{and} \det(A)\det(B)=1 \right\}\\
&\simeq & \rS(\GL (m)\times \GL (n-m))\ltimes \rM_{m, n-m}\,,
\end{eqnarray*}
where $M_{n,m}$ is the space of $n\times m$ real matrices; see Section \ref{section:notation}.
We then have $\cB =G/P$.

The $K$-invariant probability measure on $\cB$ is \textit{not} $G$-invariant. But there exists a function $j: G\times \cB\to
\R^+$ such that for all $f\in L^1(\cB)$ we have
\begin{equation}\label{eq:Jacobiant}
\int_{\cB} f(b)\, db=\int_{\cB} f(g\cdot b ) j(g,b)^n\, db \, , \qquad g\in G, \quad b \in \cB \, .
\end{equation}
We include the power $n$ to adapt our notation to \cite{OP}. By the associativity of the
action we have  $j(gg',b)=j(g,g'\cdot b )j(g',b)$ for all $g \in G$ and $b \in \cB$.
Hence, for each $\lambda\in \C$ we can define a continuous representation $\pi_\lambda$
of $G$ on $L^2(\cB)$ by
\begin{equation}\label{eq:piLambda}
[\pi_\lambda (g)f] (b ):= j(g^{-1},b )^{\lambda +n/2}f(g^{-1}\cdot b )\,, \qquad
g \in G, \quad f \in L^2(\cB), \quad \b \in \cB .
\end{equation}
A simple change of variables shows that
\[ \ip{\pi_\lambda (g)f}{h}_{L^2} =
\ip{f}{\pi_{-\overline{\lambda }}(g^{-1})h}_{L^2}\,,\qquad  g \in G\,, \quad  \, f,h \in L^2(\cB)\,.\]
In particular, $\pi_\lambda$ is unitary if and only if $\lambda$ is purely imaginary. The
representations $\pi_\lambda$ are the so-called \textit{generalized (spherical) principal series representations}
(induced from the maximal parabolic subgroup $P$), in the compact picture.
See e.g. \cite{Knapp86}, p. 169.

The representations $\pi_\lam$ can also be realized on Stiefel manifolds as follows.
According to \cite[Section  7.4.3]{Ru12},
we introduce  the radial and angular components of a matrix $x\in M_{n,m}$ of rank $m$  by
$$\rad(x)=(x^t x)^{1/2}\in \Omega, \qquad \ang (x)=x(x^t x)^{-1/2}\in \vnm,$$
so that $x=\ang (x)\,\rad(x)$. Given $\lam \in \C$, we define a mapping which assigns to every $g\in \GL(n,\bbr)$ an operator $\pi_\lam (g)$ acting on measurable functions $f$ on  $\vnm$ by the rule
\be
\pi_\lam (g) f(v)=|\mbox{\rm rad} (g^{-1}v)|^{-(\lam+n/2)} \,f(\ang (g^{-1}v)).
\ee
Clearly, $\pi_\lam (I_n)$ is an identity operator. One can prove that if $f$ is a measurable right $\rO(m)$-invariant function on  $\vnm$, then
\be
\pi_\lam (g_1g_2) f  =\pi_\lam (g_1) \, \pi_\lam (g_2) f\,, \qquad  g_1, g_2 \in \GL(n,\bbr).
\ee
The restriction of $\pi_\lam$ to $\SL(n,\R)$, acting on the space of square integrable right $\rO(m)$-invariant functions on
$\vnm$, coincides with the representation defined by (\ref{eq:piLambda}).

\subsection{The cosine transform as an intertwining operator}
\label{subsection:cosine-intertwining}

\noindent
In this section we follow the ideas in \cite{OP}. An alternative  self-contained exposition (without
 using the representation theory of semisimple Lie groups), can be
found in \cite{Ru12}.

The gain by using the representations $\pi_\lambda$ is that we now have a meromorphic family of
representations on $L^2(\cB)$ and that they are irreducible for almost all $\lam$ and closely related
to the cosine transform. For this, we
recall some results from \cite{VoganWallach}.

\begin{theorem}[Vogan-Wallach]\label{th-irredu} There exists a countable collection $\{p_n\}$ of non-zero holomorphic polynomials on $\C$ such that if $p_n(\lambda)\not= 0$ for all $n$ then $\pi_\lambda$ is irreducible. In particular, $\pi_\lambda$ is irreducible for almost all $\lambda\in \C$.
\end{theorem}
\begin{proof} This is Lemma 5.3 in \cite{VoganWallach}.\end{proof}

Let $\theta  :G \to G$ be the involutive automorphism  $\theta (g)=(g^{-1})^t$.
We remark that in \cite{OP}  notation $\CosL=\cC^{\lam-n/2}_m$ was used.

\begin{theorem}\label{th:OP1} The cosine transform intertwines $\pi_\lambda$ and $\pi_{-\lambda}\circ \theta$, namely,
\be\label {kuku} \cC^{\lam}_m\circ \pi_{\lambda + n/2} = (\pi_{-\lambda -n/2}\circ \theta)\circ \cC_m^{\lam},
\ee
whenever both sides of this equality are analytic functions of $\lam$.
\end{theorem}
\begin{proof} We refer to Theorem 2.3 and (4.10) in \cite{OP}.
\end{proof}

In fact, it is shown in \cite{OP}, Lemma 2.5 and Theorem 4.2, that $\cC^{\lam-n/2}_m=J(\lam)$, where $J(\lam)$ is a \textit{standard} intertwining operator, studied in detail among others by Knapp and Stein in \cite{KnappStein1,KnappStein2} and Vogan and Wallach in \cite{VoganWallach}. These authors show, in particular, that $\lambda \mapsto J(\lam)$ has a meromorphic extension to all of $\C$. Furthermore, Vogan and Wallach show that if $f\in C^\infty (\cB)$, then the map 
\[\{\lambda \in \C\mid \Re (\lambda )> -1+n/2\} \ni \lambda \longmapsto J(\lambda )f\in C^\infty (\cB)\]
is holomorphic.
As a consequence of $\cC_m^{\lambda-n/2}=J(\lambda)$ and \cite[1.6 Thm]{VoganWallach}, we get
the following theorem.
\begin{theorem}
The map $\lambda \mapsto \cC_m^{\lambda}$ extends meromorphically to $\C$. In particular,
for $f\in C^\infty (\cB)$ the function $\lambda \mapsto \cC_m^\lambda (f)(b)$ extends to a
meromorphic function on $\C$ and the set of possibles poles can be chosen independent of
$f$. In the complement of the singular set we have $\cC_m^\lambda (f)\in C^\infty (\cB)$.
\end{theorem}

Notice that  precise information about analiticity of more general cosine transforms, including the structure of polar sets, is presented in Theorem \ref{lhgn1} above.

The implication of (\ref{kuku})  is that $\cC^{-\lam-n/2}_m \circ \cC^{\lam-n/2}_m$
intertwines $\pi_{\lambda }$
with itself  (in the sense of meromorphic family of operators). By Theorem \ref{th-irredu} there exists a meromorphic function $\eta$  on $\C$ such that
\be \label{jjuk}
\cC_m^{-\lambda-n/2}\circ \cC_m^{\lambda-n/2}  = \eta (\lambda)\,\id_{C^\infty (\cB )}\,
\ee
for all $\lambda \in \C$ for which the left-hand side is well defined.
The shift by $n/2$ in the definition is chosen so that the final formulas agree with those in \cite{OP}
and make some formulas more symmetric. The fact that
$\eta$ is meromorphic follows by noting that
$\eta (\lambda )=\ip{\cC_m^{-\lambda-n/2}\circ \cC_m^{\lambda-n/2} (1)}{1}$.

Formula (\ref{jjuk}) is a symmetric version of (\ref{oouy}) with $\lam$ replaced by $\lam -n/2$.
The explicit value of $\eta(\lam)$ can be easily obtained from (\ref{for1na}).
An alternative, representation-theoretic method to compute the function $\eta(\lam)$, is presented in Section \ref{section:evaluation}.
The first step is the following lemma.
\begin{lemma} \label{lemma:c-eta}
Let $c(\lam)=\cC_m^{\lam-n/2}(1)$. Then $\eta(\lam)=c(\lam)c(-\lam)$.
\end{lemma}
Note that  $c(\lambda )$ is nothing but $\eta_{\rm trivial}(\lambda )$ in Theorem \ref{th:etaMu}.

\begin{remark} There are several ways to prove the meromorphic extension of the standard intertwining operators.
The proof  in \cite{VoganWallach}  uses tensoring with finite dimensional representations of $G$ to
deduce a relationship between $\cC_m^{\lambda}$ and $\cC_m^{\lambda +2n}$. In
fact, there exists a family of (non-invariant) differential operators $D_\lambda$ on $\cB$ and
a polynomial $b(\lambda)$, the Bernstein polynomial, such that
\begin{equation}\label{eq:BernsteinPol}
b (\lambda )\cC_m^\lambda (f)= \cC_m^{\lambda +2n}(D_\lambda (f))
\end{equation}
 \cite[Thm. 1.4]{VoganWallach}. 
Another way to derive an equation of the form (\ref{eq:BernsteinPol}) is
to   convert the integral defining $\cC_m^\lambda$ into an integral over the orbit
of certain nilpotent group $\bar{N}$, as usually done in the study of standard intertwining
operators, and then use the ideas from \cite{BD92, O87,OP01}.
In the case where $G/P$ is a symmetric $R$-space (which contains the case of Grassmann manifolds), the standard intertwining operators $J(\lam)$ have been recently studied by Clerc in \cite{Clerc}, using Loos' theory of positive
Jordan triple systems. In particular, Clerc explicitly computes the Bernstein polynomials $b(\lam)$  in (\ref{eq:BernsteinPol}),
and, hence, proves the meromorphic extension of $J(\lam)$ for this class of symmetric spaces.

\end{remark}

Finally, one can stick with the domain where $\lambda \mapsto \cC_m^\lambda$ is holomorphic and determine the $K$-spectrum functions $\eta_\pi (\lambda)$
in (\ref{eq:defEta}). As rational functions of $\Gamma$-factors, these functions have meromorphic extension to $\C$. Hence, $\lambda \mapsto \cC_m^\lambda$ itself has  meromorphic extension by (\ref{eq:defEta}). We will comment more on that in Remark \ref{rem:mero-cont-eta}.

\subsection{Historical remarks}
We conclude this section with a few  historical remarks. The standard intertwining operators $J(\lam)$,
as a meromorphic family of singular integral operators on $K$ or $\bar{N}$, have been central objects in the  study of representation theory of semimisimple Lie groups since the fundamental works of Knapp and Stein
\cite{KnappStein1, KnappStein2},  Harish-Chandra
\cite{HC76}, and several others.
In our case
\[\bar{
N} =\left\{ \left. \begin{pmatrix} I_m & 0\\ X & I_{n-m}\end{pmatrix}\, \right| \, X\in \rM_{m,n-m}\right\}\, .
\]
Then, in the realization of the generalized principal series representations on $L^2(\cB)$,  the kernel of $J(\lam)$ is $\Cos^{\lambda-n/2}(b,c)$. 
But in most cases there is neither an explicit formula nor  geometric interpretation of the kernel defining $J(\lam)$.

Apart of customary applications of the cosine transform in convex
geometry, probability, and the Banach space theory, similar integrals
turned up independently as standard intertwining operators between generalized principal series  representations of $\SL (n,\K)$,   where $\K=\R,\C$ or $\bbh$.

The real case was studied in \cite{DM}, the complex case in \cite{DZ1997}, and  the quaternionic case in \cite{Pa}. In these articles it was shown that integrals of the form $$\int_{\cB} |(x,y)|^{\lambda -n/2}f(x)\, dx,$$ with some modification for $\K=\C$ or $\bbh$, define intertwining operators between generalized principal series representations induced from a maximal parabolic subgroup in $\SL (n+1,\K)$. The $K$-spectrum was determined, yielding the cases of irreducibility and, more generally, the composition series of those representations. Among the applications, there were some embeddings of the complementary series and the study of the so-called canonical representations on some Riemannian symmetric spaces of the noncompact type, \cite{DHjfa,DH,DP}. However the connections of these considerations to convex geometry, to the cosine transform and to the Funk and Radon transforms was neither discussed nor mentioned.
These connections were first published in \cite{OP} in the context of the
Grassmannians over $\R$ , $\C$ and $\bbh$. However, it was probably S. Alesker who was the first to remark in the  unpublished manuscript \cite{A} that the cosine transform is a $\SL(n,\R)$-intertwining operator.
It was also shown in \cite{Zh09} that the $\SinL$-transform (a transform related to the sine transform) can be viewed as a Knapp-Stein intertwining operator. This was used to construct complementary series representations for $\GL (2n,\R )$. The
$\SinL$- transform is then also naturally linked to reflection positivity, which relates complementary series representations of $\GL (2n,\R)$ to the highest weight representations of $\SU (n,n)$, \cite{FL10,FL11,NO,JoGo1,JoGo2}. Notice, however, that the
definition of the $\SinL$-transform in \cite{Zh09} differs from the one in \cite{Ru02}, \cite{Ru12}; see also \cite{Ru02b} for the sine transform on the hyperbolic space.

\section{The spherical representations}\label{se-Grassman}
\noindent
The functions $\eta_\pi (\lambda )$ in (\ref{eq:defEta}) are parametrized by the $L$-spherical representations
of $K$. The main purpose of this section is to present this parametrization, which is given
by a semilattice in a finite dimensional Euclidean space associated with a maximal
flat submanifold of $\cB$.  We will, therefore, have to study the structure of the symmetric space $\cB$. We refer to  \cite{T94} and the books by Helgason \cite{Sig, GGA} for more detailed discussions
and proofs. To bring the discussion closer to standard references in Lie theory we also introduce some
Lie theoretical notation which we have avoided so far.

Let
\begin{eqnarray*}
\fg&=& \{X \in \rM_{n,n}\mid\tr(X)=0\}\,,\\
\fk&=& \{X \in \rM_{n,n}\mid X^t=-X\}\,,
\end{eqnarray*}
be  the Lie algebras of $G=\SL(n,\R)$ and $K=\SO(n)$, respectively. The derived involution of $\theta$ on $\fg$, still denoted $\theta$, is given by $\theta(X)=-X^t$. Hence, $\fk=\fg(1,\theta)$, the eigenspace of $\theta$ on $\fg$ with eigenvalue 1.
We fix once and for all the $G$-invariant bilinear form
$\beta(X,Y) =\frac{n}{m(n-m)}\tr (XY)$ on $\fg$. Note that $\beta $ is negative definite on $\fk$ and
$\ip{X}{Y}=-\beta (X,\theta (Y))$ is an inner product
on $\fg$ such that $\ad (X)^t=-\ad (\theta (X))$, where, as usual, 
$\ad (X)Y=[X,Y]=XY-YX$.  The normalization of $\beta$ is chosen so that it agrees with \cite{OP}.

We recall that $\cB$ is a symmetric space corresponding  to the involution
\[\tau (x)=\begin{pmatrix} \rI_m & 0 \\ 0 & -\rI_{n-m}\end{pmatrix}x\begin{pmatrix} \rI_m& 0 \\ 0 & -\rI_{n-m}\end{pmatrix}
=\begin{pmatrix} A & - B\\ - C & D\end{pmatrix} \text{ for } x=\begin{pmatrix} A & B\\ C & D\end{pmatrix}\, ,\]
where for $r\in \N$ we denote by $\rI_r$ the $r\times r$ identity matrix.
Note that $\tau$ in fact defines an involution of $G$ and that the derived involution on the Lie algebra $\fg$ is given
by the same form.

We have $\fk=\fl \oplus \fq$ where $\fl \simeq \so (m)\times \so (n-m)$ is the Lie algebra of $L$ and
\[\fq=\fk (-1,\tau)=\left\{\left. Q (X)=\begin{pmatrix} 0_{mm} & X\\ -X^t & 0_{n-m,n-m}\end{pmatrix}\, \right|\, X\in \rM_{m,n-m}
\right\}\, .\]
Let $E_{\nu,\mu}=(\delta_{i\nu}\delta_{j\mu})_{i,j}$ denote the matrix in $\rM_{m,n-m}$
with all entries equal to $0$ but the $(\nu,\mu)$-th which is equal to $1$.
For $\mathbf{t}=(t_1,\ldots ,t_m)^t\in \R^m$ we set
\begin{eqnarray*}
X (\mathbf{t})&=&-\sum_{j=1}^m t_j E_{n-2m+j,j}\,\in \rM_{m,n-m}\,,\\
Y(\mathbf{t})&= & Q(X(\mathbf{t}))         \in \fq\,.
\end{eqnarray*}
Then $\fb = \{ Y (\mathbf{t})\mid \mathbf{t}\in\R^m\}\simeq \R^m$
is a maximal abelian subspace of $\fq$.

To describe the set $\wKL$ we note first that   $\cB$ is not simply connected. So we cannot use the
Cartan-Helgason theorem \cite[p. 535]{GGA} directly, but only a slight modification is needed.  Define $\epsilon_j(Y(\mathbf{t})):=i t_j$.   We will identify the element $\lambda=\sum_{j=1}^m\lambda_j \epsilon_j
\in \fb_\C^*$
with the corresponding vector $\mathbf{\lambda
}=(\lambda_1,\ldots ,\lambda_m)$.

If $H\in \fb$, then $\ad (H)$ is skew-symmetric on $\fk$ with respect to the inner product $\ip{\, \cdot\, }{\, \cdot\, }$. Hence
$\ad (H) $ is diagonalizable over $\C$ with purely imaginary eigenvalues.  For $\alpha\in i\fb^*$ let
\[\fk_\C^\alpha :=\{X\in \fk_\C\mid (\forall H\in \fb)\; \ad (H)X=\alpha (H)X\}\]
be the joint $\alpha$-eigenspace. Let
\[\Delta_\fk =\{\alpha \in i\fb^*\mid \alpha\not=0 \text{ and }\fk_\C^\alpha \not=\{0\}\}\, .\]
The dimension of $\fk_\C^\alpha$ is called the \textit{multiplicity} of $\alpha$ (in $\fk_\C$).

\begin{lemma}
We have
\[\Delta_\fk=\{\pm \epsilon_i\pm \epsilon_j \;(1\le i\not= j\le m\,, \pm \,\, \mathrm{ independently }), \; \pm \epsilon_i \; (1\le i\le m)\, \}\]
with multiplicities respectively $1$ (and not there if $m=1$), $2n-m$ (and not there if $m=n-m$).
\end{lemma}

\begin{proof} This follows from \cite{Sig}: the table on page 518, the description of the simple root systems on page 462 ff. and the Satake diagrams on pages 532--533.
\end{proof}

We let
\[\Delta_\fk^+=\{ \epsilon_i\pm \epsilon_j \,(1\le i  < j\le m\, ), \, \epsilon_i \, (1\le i\le m)\, \}\, .\]
\begin{lemma} Let $\displaystyle \rho_\fk =\frac{1}{2}\sum_{\alpha\in\Delta_\fk^+} \dim (\fk_\C^\alpha) \alpha \in i\fb^*$.
Then
$\displaystyle\rho_\fk =\sum_{j=1}^m \Big(\frac{n}{2}-j\Big)\epsilon_j$.
\end{lemma}

Let now $(\pi,V_\pi)$ be a unitary irreducible representation of $K$. Then $V_\pi$ is finite dimensional.
Moreover, $\pi (H)=\Big. \dfrac{d}{dt}\Big|_{t=0} \pi (\exp (tH))$ is skew-symmetric, hence, diagonalizable,
for all $H\in \fb$ (in fact, $\pi (H)$ is diagonalizable for all $H\in \fk$). 
Let $\Gamma (\pi )\subset i\fb^*$ be the finite set of joint eigenvalues of $\pi(H)$ with $H \in \fb$.
For $\mu \in \Gamma (\pi)$, let $V^\mu_\pi \subset V_\pi$ denote the joint eigenspace of eigenvalue $\mu$.
If $X\in \fk_\C^\alpha$ and $v \in V_\pi^\mu$,  then $\pi (X)v\in V_\pi^{\mu +\alpha}$.
Thus, there exists a $\mu=\mu_\pi \in \Gamma (\pi )$ such that $\pi (\fk_\C^\alpha )V_\pi^\mu=\{0\}$ for all
$\alpha \in\Delta_\fk^+$. This only uses that $\pi$ is finite dimensional, but the irreducibility implies that
this $\mu$ is unique. It is called the \textit{highest weight} of $\pi$. Finally we have $\pi \simeq \sigma$ if and
only if $\mu_\pi=\mu_\sigma$.

 Let $\widetilde{K}$ be the universal  covering group of $K$. Then $\tau $ lifts to an involution $\widetilde{\tau}$
on $\widetilde{K}$, $\widetilde{L}:=\widetilde{K}^{\widetilde{\tau}}$ is connected, and
$\widetilde{\cB}:=\widetilde{K}/\widetilde{L}$ is the universal covering of $\cB$. Replacing $K$
by $\widetilde{K}$ etc., we can talk about $\widetilde{L}$-spherical representations of $\widetilde{K}$ and
their highest weights. The following theorem is a consequence of the Cartan-Helgason theorem \cite[p. 535]{GGA}.

\begin{theorem} The map $\pi\mapsto \mu_\pi$ sets up a bijection between the set of $\widetilde{L}$-spherical representations of $\widetilde{K}$ and the semi-lattice
\be\label {23da}
\Lambda^+(\widetilde{\cB}) =
\left\{\mu \in i\fb^*\, \left|\, (\forall \alpha \in \Delta_\fk^+)\, \frac{\ip{\mu}{\alpha}}{\ip{\alpha}{\alpha}}\in\Z^+
\right. \right\}\, .\ee
Furthermore, if $m=n/2$, then
\[\Lambda^+(\widetilde{\cB})=\{(\mu_1,\ldots ,\mu_m)\in \Z^m\mid
\mu_1\ge \mu_2\ge \cdots \ge \mu_{m-1}\ge |\mu_m|\}\, .\]
Otherwise, 
 \[\Lambda^+(\widetilde{\cB})=\{(\mu_1,\ldots ,\mu_m)\in \Z^m\mid
\mu_1\ge \mu_2\ge \cdots \ge \mu_{m-1}\ge \mu_m\ge 0\}\, .\]
\end{theorem}

If $\mu \in \Lambda^+ (\widetilde{\cB})$, then we write $(\pi_\mu,V_\mu)$ for the corresponding $\widetilde{L}$-spherical
representation. Recall the notation $\Phi_{\pi_\mu}$ from (\ref{eq:EmbL2}). Let $\Lambda^+ (\cB)$ denote the sublattice in $ \Lambda^+ (\widetilde{\cB})$ which corresponds to
$L$-spherical representations of $K$. Then $\mu\in\Lambda^+(\cB)$ if and only if the functions
$\Phi_{\pi_\mu}(v)$, which are originally defined on $\widetilde{\cB}$, factor to functions on
$\cB$. For that, let $v\in V_\mu^\mu$ and $H\in\fb$. We can normalize $v$ and $e_{\pi_\mu}$ so that 
\[\Phi_{\pi_\mu}(v;\exp H) =e^{\mu (H)}\, .\]
The same argument as for the sphere \cite[Ch. III.12]{T94},  proves  the following theorem.
\begin{theorem}\label{le:Lgrassmanian}  If $m=n-m$, then
\[\LB =\{\mu=\sum_{j=1}^m \mu_j\e_j\mid \mu_j\in 2\N_0 \text{and} \mu_1\ge \ldots \ge \mu_{m-1}\ge |\mu_m| \;\}\, .\]
In all other cases, $$\LB =\{\mu=\sum_{j=1}^m \mu_j\e_j\mid \mu_j\in 2\N_0 \text{and} \mu_1\ge \ldots \ge \mu_m\ge 0 \;\}\,.$$
\end{theorem}

\section{The generation of the $K$-spectrum}\label{section:evaluation}
\noindent
Recall from Section \ref{se-Grassman} the involution $\theta(X)=-X^t$ on $\fg$. The Lie algebra $\fg$ decomposes into eigenspaces of $\theta$ as $\fg=\fk\oplus\fs$, where
\begin{equation*}
\fs=\fg (-1, \theta )=\{X\in\rM_{n,n} \mid \theta(X)=-X \textrm{ and }\Tr (X)=0\}\, .
\end{equation*}
Then, except in the case $n=2$, the complexification $\fs_\C$ of $\fs$ is an  irreducible $L$-spherical representation of $K$.
For $n=2$ this representation decomposes into two one-dimensional representations.

Let
\[H_o=\begin{pmatrix} \frac{n-m}{n}\,\rI_m & 0\cr 0 &
-\frac{m}{n}\, \rI_{n-m}\end{pmatrix}\in \fs\, .\]
Then $H_o$ is $L$-fixed and $\langle H_0,H_0 \rangle=1$. Define
$\fa :=\R H_o$.
The operator $\ad (H_0)$ has spectrum $\{0,1,-1\}$ and
$\fn= \fg (1, \ad (H_0))$.

Let $\Ad(k)$ denote the conjugation by $k$.
Define a map $\omega :\fs_\C \to C^\infty (\cB)$ by
\[\omega (Y)(k):= \langle Y,\Ad (k)H_o\rangle =\beta(Y,\Ad (k)H_o)=\tfrac{n}{m(n-m)}\Tr (YkH_ok^{-1}) \]
and note that
\[\omega (\Ad (h)Y)(k)=\langle \Ad (h)Y,\Ad (k)H_o\rangle = \langle Y, \Ad (h^{-1}k)H_o\rangle
=\omega (Y)(h^{-1}k)\, .\]
Thus $\omega $ is a $K$-intertwining operator.

Fix an orthonormal basis  $X_1,\ldots ,X_{\dim \fq}$ of $\fq$ such that $X_1,\ldots ,X_m$, is an orthonormal basis of $\fb$. Denote by $\Omega =- \sum_j X_j^2$ the corresponding positive definite Laplace operator on $\cB$.  Then
\[\Omega |_{L^2_\mu (\cB )} =\omega(\mu)\, \id\, ,\]
where
$$\omega(\mu)= \langle \mu +2\rho_\fk ,\mu\rangle\,.$$
A simple calculation then gives:
\begin{lemma}\label{le:omMu}
Let $\mu=(\mu_1,\dots,\mu_m)  \in \LB$. Then
\[ \omega (\mu )=\frac{m(n-m) }{2n}\sum_{j=1}^m \Big(\mu_j^2 +\mu_j (n-2j)\Big)\, .\]
\end{lemma}

For $f\in C^\infty (\cB)$ denote by $M (f ) : L^2(\cB)\to L^2(\cB)$ the multiplication operator $g\mapsto fg$.
Recall the notation $\pi_0$ for the finite dimensional spherical representation of highest weight $0\in \LB$.
\begin{theorem}\label{th:Spgen} Let $Y\in \fs$. Then $[\Omega,M(\omega (Y))]=2\pi_0(Y)$.
\end{theorem}
\begin{proof} This is Theorem 2.3 in \cite{BOO}.\end{proof}

For $\mu \in\LB$ define $\Psi_\mu :  L^2_\mu (\cB)\otimes \fs_\C \to L^2(\cB)$ by
\[\Psi_\mu ( \varphi \otimes Y):= M(\omega (Y))\varphi\, .\]
Observe that for $k \in K$, $Y \in \fs_\C$, and $\varphi\in L^2_\mu(\cB)$ we have
\[ \ell(k)\big( M(\omega(Y)\varphi\big)=\big(\ell(k)\omega(Y)\big)(\ell(k)\varphi)=M\big(\omega(\Ad(k)Y)\big)(\ell(k)\varphi) \]
with $\Ad(k)Y \in \fs_\C$ and $\ell(k)\varphi \in L^2_\mu(\cB)$. Hence, $\Psi_\mu$ is $K$-equivariant and $\Im \Psi_\mu$ is $K$-invariant.
Define a finite subset $S(\mu)\subset \LB$ by
\[\Im \Psi_\mu \simeq_K \bigoplus_{\sigma \in S(\mu )} L^2_\sigma (\cB)\, .\]

\begin{lemma} Let $\mu \in \LB$. Then
 \[S (\mu )=\{\mathbf{\mu} \pm 2\epsilon_j\mid j=1,\ldots , m \}\cap \LB\, .\]
These representations occur with multiplicity one.
\end{lemma}

Denote by $\pr_\sigma $ the orthogonal projection $L^2(\cB)\to L^2_\sigma (\cB)$. The first spectrum generating
relation which follows from Theorem \ref{th:Spgen}, see also \cite[Cor. 2.6]{BOO},  states:

\begin{lemma} Assume that $\mu\in \LB$. Let $\sigma\in S(\mu)$, $Y\in \fs_\C$, and $\lambda\in \C$. Let
\begin{equation}
\label{eq:omegamusigma}
\omega_{\sigma\mu} (Y) : = \pr_\sigma\circ M( \omega (Y))|_{L^2_\mu (\cB)} :
L^2_\mu (\cB)\to L^2_\sigma (\cB)\, .
\end{equation}
Then
\begin{equation}\label{eq-spectGen1}
\pr_\sigma\circ \pi_{\lam} (Y)|_{L^2_\mu (\cB)}=\frac{1}{2}(\omega (\sigma )-\omega (\mu )+2\tfrac{m(n-m)}{n}\, \lam)\omega_{\sigma\mu} (Y) \, .\end{equation}
\end{lemma}

The spectrum generating relation that we are looking for can now easily be deducted and we get:

\begin{lemma}\label{le-eta1}
Let $\mu =(\mu_1,\ldots ,\mu_m)\in \LB$ and $\lambda\in\C$. Then
\begin{equation}\label{eq-eta1}
\frac{\eta_{\mu +2\epsilon_j}(\lambda )}{\eta_\mu (\lambda )}=\frac{\lambda -\mu_j +j-1}{\lambda  + \mu_j + n -j+1}=-\frac{-\lambda + \mu_j  - j+1}{\lambda  + \mu_j + n -j+1}
\end{equation}
and $\eta_0 (\lambda )=c(\lambda )$.
\end{lemma}
\begin{proof}  First we apply
 $\cC^{\lambda-n/2}_m$ to (\ref{eq-spectGen1}) from the left, using that $\cC^{\lambda-n/2}_m$ commutes with $\pr_\sigma$ and that $\cC^{\lambda-n/2}_m\circ \pi_\lambda (Y)=\pi_{-\lambda}\circ \theta (Y)\circ \cC^{\lambda-n/2}_m= - \pi_{-\lambda } (Y)\circ \cC_m^{\lambda-n/2}$. We then get:
\begin{multline*}
\big(\omega (\sigma )-\omega (\mu )+2\tfrac{m(n-m)}{n}\, \lam\big)\eta_\sigma (\lambda-n/2 )\omega_{\sigma \mu}(Y)
=\\
-\big(\omega (\sigma )-\omega (\mu )-2\tfrac{m(n-m)}{n}\, \lam\big)\eta_\mu (\lambda-n/2)\omega_{\sigma\mu}(Y)\, .
\end{multline*}
As $\omega_{\sigma\delta}(Y)$ is non-zero, for generic $\lambda$ it can be canceled out.

Now insert the expression from Lemma \ref{le:omMu} to get
\[\omega (\mu +2\epsilon_j)-\omega (\mu)=\tfrac{2m(n-m)}{n}\left(\mu_j +n/2 -(j-1)\right)\]
and the claim follows.
 The last statement follows from the fact that
$\pi_\lambda$ is irreducible for generic $\lambda$, hence, iterated application of  (\ref{eq:omegamusigma}) will in the end reach all $K$-types starting from the trivial $K$-type.
\end{proof}

Lemma \ref{eq-eta1} tells us that the evaluation of $\eta_\mu (\lambda )$ can be done in two steps. First
we determine the function $\eta_0(\lambda) $ and then use (\ref{eq-eta1}) as an inductive procedure to determine the rest. The final result is given in the following theorem.
 It is presented in terms of  $\Gamma$-functions associated to the cone $\Omega$ of $m\times m$ positive definite matrices, namely,
\begin{equation}\label{def-Gamma}
\Gam_{\Omega}  (\lam) =\pi^{m(m-1)/4} \prod_{j=1}^m \Gamma (\lambda_j -(j-1)/2),\qquad \lam=(\lam_1,\ldots, \lam_m)\in\bbc^m.
\end{equation}
This integral is a generalization of $\Gam_{m}  (\lam)$ in (\ref{2.4}); cf. \cite [p. 123] {FK}, \cite [Sec. 2.2]{Ru12}.
In the following the scalar parameters, which occur in the argument of $\Gam_{\Omega}$,  are interpreted as vector valued, for instance,
$n\sim (n, \ldots ,n)$, $\lam \sim (\lambda, \ldots ,\lambda)$.

\begin{theorem}[\cite{OP}]\label{th:etaMu}  Let $\Lambda^+ (\cB)$ be the sublattice in Theorem \ref{le:Lgrassmanian}  parametrizing the $L$-spherical representations of $K$, let
 $\mu =(\mu_1,\ldots ,\mu_m)\in \LB$,  and $\lambda\in\C$. Then the $K$-spectrum of the cosine transform $\cC^\lam_m$  is given by:
\begin{equation}
\eta_\mu (\lambda)=(-1)^{|\mu|/2} \; \frac{\Gamma_{m}\left(n/2\right)}{\Gamma_{m} \left(m/2\right)}\,
\frac{\Gamma_{m}\left((\lambda+m)/2)\right)}{\Gamma_{m} \left(-\lambda/2\right)}\,\frac{\Gamma_{\Omega}\left((\mu-\lambda)/2\right)}{\Gamma_{\Omega} \left((\lambda+n+\mu)/2\right)}\, .
\end{equation}
\end{theorem}

\begin{remark} Owing to (\ref{n0mby}), the spectrum of the normalized cosine transform $\Cs^\lam_m$  has the simpler form
\begin{equation}
\tilde \eta_\mu (\lambda)=(-1)^{|\mu|/2} \;
\frac{\Gamma_{\Omega}\left((\mu-\lambda)/2\right)}{\Gamma_{\Omega} \left((\lambda+n+\mu)/2\right)}\, .
\end{equation}
In the case $m=1$ this formula coincides with (\ref{55}).
\end{remark}

\begin{remark}
\label{rem:mero-cont-eta}
In Section \ref{se:ConnRep} we referred to the result of Vogan and Wallach on the meromorphic continuation of the intertwining operator $J(\lambda )$. This result is not needed for the computation of $\eta_\mu (\lambda )$. Indeed, it is enough to know that $J(\lambda )$ is holomorphic on some open subset of $\C$ as that is all what is needed to determine
$\eta_\mu (\lambda )$ in Theorem \ref{th:etaMu}. We can then extend $\cC_m^\lambda$ meromorphically on each
$K$-type. Note, however, that this is weaker than the statement in \cite{VoganWallach} which extends $\cC_m^\lambda f$ for all
smooth functions.
\end{remark}

\bibliographystyle{plain}
\bibliography{OPR-bib}

\end{document}